\documentclass[12pt]{amsart}
\usepackage{amssymb,amsmath}

\def\pt{\partial_t}
\def\H {{\mathcal H}}
\def\C {{\mathcal C}}
\def\H {{\mathcal H}}
\def\M {{\mathcal M}}
\def\B {{\mathcal B}}
\def\Q {{\mathcal Q}}
\def\A {{\mathcal A}}
\def\R {\mathbb{R}}
\def\N {\mathbb{N}}
\def\D {{\mathfrak D}}
\def\HH{{\rm H}}
\def\eps{\varepsilon}
\def\e{{\rm e}}
\def\oo{{\nu}}
\def\d{{\rm d}}

\def\ddt{\frac{\d}{\d t}}
\def \l {\langle}
\def \r {\rangle}
\def \and {{\qquad\text{and}\qquad}}
\def\cbt {{\sc cbt}}
\newtheorem{proposition}{Proposition}[section]
\newtheorem{theorem}[proposition]{Theorem}
\newtheorem{corollary}[proposition]{Corollary}
\newtheorem{lemma}[proposition]{Lemma}
\theoremstyle{definition}
\newtheorem{definition}[proposition]{Definition}
\newtheorem{remark}[proposition]{Remark}

\numberwithin{equation}{section}
\def \au {\rm}
\def \ti {\it}
\def \jou {\rm}
\def \bk {\it}
\def \no#1#2#3 {{\bf #1} (#3), #2.}
\def \eds#1#2#3 {#1, #2, #3.}

\title[Viscoelasticity with singularly oscillating external forces]
{Averaging of equations of viscoelasticity\\
with singularly oscillating external forces}

\author[V.V. Chepyzhov, M. Conti and V. Pata]
{Vladimir V. Chepyzhov, Monica Conti and Vittorino Pata}

\address{Institute for Information Transmission Problems, RAS
\newline\indent
Bolshoy Karetniy 19, Moscow 101447, Russia
\newline\indent
National Research University Higher School of Economics
\newline\indent
Myasnitskaya Street 20, Moscow 101000, Russia}
\email{chep@iitp.ru}

\address{Politecnico di Milano - Dipartimento di Matematica
\newline\indent
Via Bonardi 9, 20133 Milano, Italy}
\email{monica.conti@polimi.it}
\email{vittorino.pata@polimi.it}

\thanks{The first author VVC has been partially supported by the {\it Russian Foundation of Basic Researches} (project 14-01-0034).}

\begin{document}

\begin{abstract}
Given $\rho\in[0,1]$, we
consider for $\eps\in(0,1]$ the nonautonomous viscoelastic
equation with a singularly oscillating external force
$$
\partial_{tt} u-\kappa(0)\Delta u - \int_0^\infty
\kappa'(s)\Delta u(t-s)\d s +f(u)=g_{0}(t)+\eps ^{-\rho }g_{1}(t/\eps )
$$
together with the {\it averaged} equation
$$
\partial_{tt} u-\kappa(0)\Delta u - \int_0^\infty
\kappa'(s)\Delta u(t-s)\d s +f(u)=g_{0}(t).
$$
Under suitable assumptions on the nonlinearity and on the external force,
the related solution processes $S_\eps(t,\tau)$ acting on the
natural weak energy space $\H$
are shown to possess uniform attractors $\A^\eps$.
Within the further assumption $\rho<1$, the family
$\A^\eps$ turns out to be bounded in $\H$, uniformly with respect to $\eps\in[0,1]$.
The convergence of the attractors $\A^\eps$
to the attractor $\A^0$ of the averaged equation
as $\eps\to 0$ is also established.
\end{abstract}

\maketitle

\section{Introduction}
\label{intro}

\noindent
Let $\Omega\subset\R^3$ be a bounded domain with smooth boundary $\partial\Omega$,
and let $\rho\in[0,1]$ be a fixed parameter.
For every $\eps\in[0,1]$ and any given $\tau\in\R$,  we consider for $t>\tau$
the hyperbolic equation with memory,
arising in the theory of isothermal viscoelasticity~\cite{FM,RHN},
in the unknown
$u=u(x, t): \Omega \times \R \to \R$
\begin{equation}
\label{BASE}
\partial_{tt} u
-\kappa(0)\Delta u - \int_0^\infty
\kappa'(s)\Delta u(t-s)\d s +f(u)=g^{\eps}(t),
\end{equation}
where
$$g^{\eps}(x,t)=
\begin{cases}
g_{0}(x,t)+\eps^{-\rho }g_{1}(x,t/\eps) & \text{if } \eps>0,\\
g_{0}(x,t) & \text{if } \eps=0.
\end{cases}
$$
The equation is supplemented with the Dirichlet boundary
condition
\begin{equation}
\label{bd-cond-u}
u(x,t)_{|x\in\partial\Omega}=0.
\end{equation}
The variable $u$, describing
the displacement field relative to the reference
configuration of a viscoelastic body occupying the volume
$\Omega$ at rest, is interpreted as an initial datum for $t\leq\tau$, namely,
\begin{equation}
\label{in-cond-u}
\begin{cases}
u(\tau) = u_\tau,\\
\pt u(\tau) = v_\tau,\\
u(\tau-s) = q_\tau(s),\quad s>0,
\end{cases}
\end{equation}
where $u_\tau$, $v_\tau$ and the function $q_\tau$ are assigned data.
The function $\kappa$, usually called memory kernel,
is supposed to be convex, decreasing and such that
$$\kappa(0)>\kappa(\infty)>0.$$
Without loss of generality, we will assume hereafter $\kappa(\infty)=1$.
Notably, in the present model, the dissipation mechanism
is entirely contained in the memory term, which provides a very weak
form of damping, whereas no instantaneous friction is active.
The term $f:\R\to\R$ is a nonlinear function of the displacement
having subcubic growth, and complying with rather standard dissipativity conditions.
Physically relevant examples of functions in this class
are
$$f(u)=a |u|^{p-1}u,\quad a>0,\,p\in[1,3),$$
appearing in the equation of relativistic quantum mechanics,
and
$$f(u)=b\sin u,\quad b>0,$$
yielding a sine-Gordon model describing a
Josephson junction driven by a current source (see e.g.\
\cite{HAL,TEMbook} and references therein).
Finally,
$g^{\eps}(t)$
represents a singularly (if $\rho>0$) oscillating external force.

\smallskip
The aim of the present paper is to study the asymptotic properties of~\eqref{BASE}-\eqref{in-cond-u} depending on the
parameter $\eps$, which represents the
(time) oscillation rate in the external force, whose
amplitude is of order $\eps ^{-\rho }$. To this end,
following a pioneering idea of C.M.\ Dafermos~\cite{DAF},
we first translate the initial-boundary value problem above in the so-called
past history framework.
Accordingly, we introduce for $t \geq \tau$ the past history variable
$$
\eta^t(s) = u(t) - u(t-s), \quad s >0.
$$
Defining the (positive and summable) kernel
$$\mu(s)=-\kappa'(s),$$
where the {\it prime} stands for derivative with respect to $s$,
equation~\eqref{BASE} is rewritten
as the system of equations for $t>\tau$
\begin{equation}
\label{MEDIUM}
\begin{cases}
\partial_{tt} u-\Delta u
\displaystyle-\int_0^\infty \mu(s) \Delta\eta(s)\d s+f(u)= g^{\eps}(t),\\
\pt \eta=-\partial_s\eta+\pt u,
\end{cases}
\end{equation}
in the unknown variables $u=u(t)$ and $\eta=\eta^t(\cdot)$,
subject to the boundary conditions
\begin{equation}
\label{MEDIUMBCaa}
u(t)_{|\partial\Omega}=0
\and
\eta^t_{|\partial\Omega}=0,
\end{equation}
with $\eta$ complying with the further constraint
\begin{equation}
\label{MEDIUMBC}
\eta^t(0)=\lim_{s\to 0}\eta^t(s)=0.
\end{equation}
In turn,
the initial conditions become
\begin{equation}
\label{MEDIUMCI}
\begin{cases}
u(\tau) = u_\tau, \\
\pt u(\tau) = v_\tau, \\
\eta^\tau = \eta_\tau,
\end{cases}
\end{equation}
having set
$$\eta_\tau(s)=u_\tau-q_\tau(s).$$
The advantage of the new formulation is that
the nonautonomous problem~\eqref{MEDIUM}-\eqref{MEDIUMCI}
generates, at any fixed $\eps$, a dynamical process $S_\eps(t,\tau)$
acting on a suitable phase space $\H$.
According to the well-established theory of nonautonomous dynamical systems~\cite{CVbook,HARbook}, the longterm
dynamics can be conveniently described in terms of (uniform) global attractors
$\A^{\eps}$ of the corresponding processes.
Indeed, our main purpose is to investigate the properties of the family $\A^{\eps}$, in dependence
of the parameter $\eps\in[0,1]$. First, within the restriction $\rho<1$,
and under suitable translation-compactness assumptions on the external forces, we
prove the uniform (with respect to $\eps$) boundedness
of the global attractors, namely,
$$
\sup_{\eps\in[0,1]}\|\A^\eps\|_\H<\infty.
$$
This fact is not at all intuitive, since
in principle the blow up of the oscillation amplitude
might overcome the averaging effect due to the scaling
$t/\eps$ appearing in $g_1$. Next, we establish
a convergence result for $\A^{\eps}$ in the limit $\eps\to 0$. More precisely,
we show that
$$
\lim_{\eps\to 0}\,
\mathrm{dist}_{\H}\big(\A^{\eps },\A^{0}\big)=0,
$$
where
$\mathrm{dist}_{\H}$
denotes the standard Hausdorff semidistance
in $\H$.
This allows us to interpret the {\it averaged} case $\eps=0$ as the formal limit
of~\eqref{BASE} as $\eps\to 0$.

\smallskip
The averaging of global attractors of nonautonomous evolution equations
in presence of nonsingular time oscillations (i.e.\ when $\rho=0$) has been
studied by several authors. See e.g.\ \cite{ChGoVi,CVbook,CV5,CVW,EfZe1,EfZe2,FiVi1,FiVi2,I1,ChVi,Ze}.
The more challenging singular case $\rho>0$ is treated in the more recent papers~\cite{CV3,CV4,CPV-aver,CPV-NS,CV2}.
In particular, closely related to our work, in~\cite{CPV-aver,CV2}
the same kind of analysis is carried out for the weakly damped wave equation
$$\partial_{tt} u-\Delta u
+\pt u+f(u)= g^{\eps}(t),
$$
corresponding to \eqref{BASE}
with the instantaneous damping $\pt u$
in place of the memory term, expressed by the convolution integral.
Actually, the presence of the memory in the model introduces essential difficulties from the very beginning
of the asymptotic analysis, namely, at the level of absorbing sets.
Indeed, at any fixed $\eps$, the
existence of an absorbing set for the process $S_\eps(t,\tau)$ generated by the nonautonomous
equation of viscoelasticity has not been established before, and requires the
use of a novel Gronwall-type lemma with parameters from~\cite{Patonw}.
A second difficulty is to obtain the uniform boundedness of the attractors $\A^\eps$.
For the damped wave equation,
the main idea of~\cite{CPV-aver} was to decompose the solution, by introducing a linear problem in order to isolate
the oscillations in a suitable way. The same ingredient is needed here, but it is not enough,
and the desired conclusion follows from a quite delicate recursion argument.
Such a uniform boundedness is crucial to prove the convergence $\A^\eps\to\A^0$.

\begin{remark}
Setting $\bar\kappa(s)=\kappa(s)-1$, an integration by parts allows us to rewrite~\eqref{BASE}
in the form
$$
\partial_{tt} u
-\Delta u - \int_0^\infty
\bar\kappa(s)\Delta u_t(t-s)\d s +f(u)=g^{\eps}(t).
$$
Thus, in the limit case when $\bar\kappa$ converges to the Dirac mass at $0^+$,
we recover the so-called strongly damped wave equation
$$
\partial_{tt} u
-\Delta u - \Delta u_t+f(u)=g^{\eps}(t),
$$
for which the whole analysis of this work applies
(although working directly with such an equation is much easier).
\end{remark}

\subsection*{Plan of the paper}
In the next \S\ref{SNot}
and \S\ref{SGen}, we introduce the notation and the general assumptions.
The generation of an $\eps$-family
of processes $S_\eps(t,\tau)$, acting on a suitable phase space $\H$, is discussed in \S\ref{SDyn}.
In \S\ref{SDis}, we study the dissipativity properties of such a family, proving
the existence of bounded absorbing sets, while in \S\ref{SUni} we show that
$S_\eps(t,\tau)$ possesses the uniform global attractor $\A^\eps$,
for every fixed $\eps\in[0,1]$.
The subsequent \S\ref{SAux} is
devoted to an auxiliary linear viscoelastic equation  with  oscillating external force.
This will be the crucial tool used in \S\ref{SBdd}, where a uniform (with respect to $\eps$) bound for the attractors
$\A^\eps$ is established. The main result on the convergence $\A^\eps\to\A^0$
as $\eps\to 0$ is stated and proved
in the final \S\ref{SCon}.

\section{Notation}
\label{SNot}

\subsection*{General agreement}
Throughout the paper, the symbols $c>0$ and $\Q(\cdot)$
will stand for a {\it generic} constant and a {\it generic} increasing positive function,
both independent of $\eps$ and $\tau$, as well
of $g_{0},g_{1}$.

\smallskip
Introducing the Hilbert space of square summable functions
on $\Omega$
$$\HH=L^2(\Omega)$$
with inner product $\l \cdot,\cdot \r$ and norm
$\|\cdot\|$,
we consider the Laplace-Dirichlet operator on $\HH$
$$A=-\Delta\qquad\text{with domain}\qquad \D(A)=H^1_0(\Omega)\cap H^2(\Omega),$$
and we define for $\sigma\in\R$ the scale of compactly nested Hilbert spaces
$$\HH^\sigma=\D(A^{\sigma/2})$$
endowed with the standard inner products and norms
$$\langle u,v\rangle_\sigma=\langle A^{\sigma/2}u,A^{\sigma/2}v\rangle,\qquad
\|u\|_{\sigma }=\|A^{\sigma/2}u\|.$$
The index $\sigma$ will be always omitted whenever zero. In particular,
we have the equalities
$$\HH^{-1}=H^{-1}(\Omega),\quad \HH^{1}=H_{0}^{1}(\Omega),\quad \HH^{2}=\D(A).$$
The symbol $\l \cdot,\cdot \r$ will also be used for the duality product between $\HH^\sigma$ and its
dual $\HH^{-\sigma}$.
Then, we introduce the $L^2$-weighted spaces on $\R^+=(0,\infty)$
$$\M^\sigma=L^2_{\mu}(\R^+,\HH^{\sigma+1})$$
normed by
$$\|\eta\|_{\M^\sigma}=\bigg(\int_0^\infty \mu(s)\|\eta(s)\|_{\sigma+1}^2\d s\bigg)^\frac12,$$
along with
the infinitesimal generator of the
right-translation semigroup on $\M$
$$T=-\partial_s
\qquad\text{with domain}\qquad
\D(T)=\big\{\eta\in\M:\,
\partial_s\eta\in \M,\,\,\eta(0)=0\big\},$$
where $\partial_s$ is the distributional derivative
with respect to the internal variable $s$.
Finally, we define the {\it extended memory spaces} (again, $\sigma$ is omitted if zero)
$$\H^\sigma=\HH^{\sigma+1}\times \HH^\sigma\times \M^\sigma$$
with the Euclidean product norm
$$\|(u,v,\eta)\|_{\H^\sigma}^2=\|u\|^2_{\sigma+1}
+\|v\|^2_\sigma+\|\eta\|^2_{\M^\sigma}.
$$
In what follows,
for any
$U=(u,v,\eta)\in\H$, we agree to call
\begin{equation}
\label{phii}
\Phi(U)=\frac12\|U\|^2_{\H}+\|u\|_{L^{p+1}}^{p+1},
\end{equation}
where $p\in[1,3)$ is the growth order of $f$. This quantity
is finite due to the Sobolev embedding $L^{p+1}(\Omega)\subset \HH^1$.
Besides, for any bounded set $\B\subset \H$, we use the notation
$$\Phi(\B)=\sup_{U\in \B}\Phi(U).$$

\section{General Assumptions}
\label{SGen}

\subsection{Assumptions on the nonlinearity}
Let $f\in \C^1(\R)$, with $f(0)=0$,
satisfy for a fixed $p\in[1,3)$ the growth and the dissipation conditions
\begin{align}
\label{crescita}
&|f'(u)|\leq c(1+|u|^{p-1}),\\
\label{dissip-phi}
&f(u)u\geq d_0|u|^{p+1}-c,
\end{align}
for some $d_0>0$.
Defining for every $u\in\HH^1$
$${\mathcal F}(u)=\int_\Omega \bigg(\int_0^{u(x)} f(y)\d y \bigg)\d x,$$
it is readily seen from \eqref{crescita}-\eqref{dissip-phi} that there exists $d>0$ such that
\begin{align}
\label{CTRLF}
d\|u\|^{p+1}_{L^{p+1}}-c &\leq {\mathcal F}(u)\leq c\|u\|^{p+1}_{L^{p+1}}+c,\\
\label{CTRLf}
d\|u\|^{p+1}_{L^{p+1}}-c &\leq \l f(u),u\r.
\end{align}
Besides, the following inequality holds:
\begin{equation}
\label{utile}
\|f(u)\|_{L^{6/5}}\leq c+ c|{\mathcal F}(u)|^{\frac{p}{p+1}}.
\end{equation}
Indeed, from \eqref{crescita}, the H\"older inequality and \eqref{CTRLF},
$$
\|f(u)\|_{L^{6/5}}
\leq c+c\left[\int_\Omega|u(x)|^{\frac{6p}5}\d x\right]^{\frac56}
\leq c +c\|u\|_{L^{p+1}}^p\leq
c+c|{\mathcal F}(u)|^{\frac{p}{p+1}}.
$$

\begin{remark}
\label{rem-lip}
In the Lipschitz case, i.e.\ when \eqref{crescita} holds with $p=1$, instead of \eqref{dissip-phi} it
is sufficient to require the weaker dissipation condition
$$\liminf_{|u|\to\infty}\frac{f(u)}{u}>-\lambda_1,$$
where $\lambda_1>0$ is the first eigenvalue of $A$.
Indeed, on account of the Poincar\'e inequality, it is a standard matter to verify
that \eqref{CTRLF}-\eqref{CTRLf} continue to hold if we redefine
$f(u)$ as $f(u)+\lambda u$,
for a suitable $\lambda<\lambda_1$ sufficiently close to $\lambda_1$,
replacing the term $Au$ in the first equation
of~\eqref{MEDIUM} with $(A-\lambda)u$.
Observe that the powers $(A-\lambda)^{\sigma/2}$ generate
the same spaces $\HH^\sigma$ with equivalent norms.
\end{remark}

\subsection{Assumptions on the external force}
The functions $g_{0}$ and
$g_{1}$
are translation bounded in $L_{\rm loc}^2(\R;\HH)$, i.e.\
\begin{align}
\| g_{0}\|_{{\rm tb}}^2& :=\sup_{t\in \R
}\int_{t}^{t+1}\|g_{0}(y)\|^2\d y= M_{0},
\label{s1r5} \\
\| g_{1}\|_{{\rm tb}}^2& :=\sup_{t\in \R
}\int_{t}^{t+1}\|g_{1}(y)\|^2 \d y = M_{1},
\label{s1r6}
\end{align}%
for some $M_{0},M_{1}\geq 0$. A straightforward consequence of (\ref{s1r6}) is
$$
\int_{t}^{t+1}\|g_{1}(y/\eps )\|^2\d y=\eps
\int_{t/\eps }^{(t+1)/\eps }\|g_{1}(y)\|^2\d y\leq
\eps (1+1/\eps)M_{1}\leq 2M_{1},
$$
so that
$$
\| g_{1}(\cdot /\eps )\|_{{\rm tb}}^2
\leq 2M_{1},\quad \forall
\eps \in (0,1].
$$
Hence, for $\eps>0$,
$$
\| g^{\eps }\|_{{\rm tb}}^2\leq
2\| g_{0}\|_{{\rm tb}}+2\eps ^{-2\rho }\| g_{1}(\cdot
/\eps )\|_{{\rm tb}}\leq 2M_{0}+4M_{1}\eps ^{-2\rho }.
$$
As a result, if we set
\begin{equation}
\label{Q}
Q_\eps
=\begin{cases}
2M_{0}+4M_{1}\eps ^{-2\rho } & \text{if } \eps>0,\\
M_{0} & \text{if } \eps=0,
\end{cases}
\end{equation}
we learn that
\begin{equation}
\label{s1BOUND}
\| g^{\eps }\|_{{\rm tb}}^2
\leq Q_\eps,\quad\forall\eps\in [0,1],
\end{equation}
meaning that the norm $\| g^{\eps }\|_{{\rm tb}}$
can grow with a rate of order $\eps ^{-\rho }$ as $\eps \to 0$.

\subsection{Assumptions on the memory kernel.} The kernel $\mu(s)=-\kappa'(s)$ is supposed to be
nonnegative, absolutely continuous and summable on $\R^+$, of total mass
$$\kappa_0:=\int_0^\infty\mu(s)\d s\in (0,1).$$
Moreover, we assume the existence of $\delta>0$ such that
\begin{equation}
\label{K2}
\mu'(s)+\delta\mu(s)\leq 0
\end{equation}
for almost every $s\in\R^+$.
It is worth noting that $\mu$ can be (weakly) singular at the origin.
The typical example of a kernel in this class is
$$\mu(s)=C s^{-\alpha}\e^{-\delta s},\quad \alpha\in[0,1),$$
for any positive constant
$$C<\frac{\delta^{1-\alpha}}{\Gamma(1-\alpha)},$$
where $\Gamma$ is the Euler-Gamma function.

\section{The Dynamical Processes}
\label{SDyn}

\noindent
As anticipated in the Introduction, the original problem \eqref{BASE}-\eqref{in-cond-u}
can be translated into the evolution system in the unknown variables
$u=u(t)$ and $\eta=\eta^t(\cdot)$
\begin{equation}
\label{PROBLEM}
\begin{cases}
\partial_{tt} u+Au
+\displaystyle\int_0^\infty \mu(s) A\eta(s)\d s+f(u)= g^\eps(t),\\
\pt \eta=T\eta+\pt u,\\
\noalign{\vskip1.5mm}
u(\tau)=u_\tau,\quad \pt u(\tau)=v_\tau,\quad \eta^\tau=\eta_\tau,
\end{cases}
\end{equation}
where the set of data
$$(u_\tau,v_\tau,\eta_\tau)\in\H$$
is assigned at an arbitrary
initial time $\tau \in \R$.
The equivalence between the two formulations
is discussed in~\cite{Terreni}.
Introducing the three-component vectors
$$U(t)=(u(t),\pt u(t),\eta^t)
\and
U_\tau=(u_\tau,v_\tau,\eta_\tau),
$$
we view \eqref{PROBLEM} as the semilinear
ODE in $\H$
\begin{equation}
\label{SYSABS}
\begin{cases}
\displaystyle
\ddt U(t)={\mathbb A} U(t)+{\mathbb F}_\eps (U(t),t),\\
\noalign{\vskip1.5mm}
U(\tau)=U_\tau,
\end{cases}
\end{equation}
where ${\mathbb A}$ is the linear operator on $\H$ acting on the vector $U=(u,v,\eta)$ as
$$
{\mathbb A} U=
\big(v,
- A\bigg[u +\int_0^\infty \mu(s) \eta(s)\d s\bigg],
T\eta + v
\big)
$$
with domain
$$
\D({\mathbb A}) = \bigg\{U\in\H :\,\,
v\in \HH^1,\,\,
u +\int_0^\infty \mu(s) \eta(s)\d s \in \HH,\,\,
\eta \in \D(T)\bigg\},
$$
while
$${\mathbb F}_\eps(U,t)=\big(0,g^\eps(t)-f(u),0\big).
$$ From the same paper~\cite{Terreni} (but see also~\cite{ViscoCP,GMPZ}),
it is well-known that for every fixed $\eps\in[0,1]$
and every $U_\tau\in\H$
the initial value
problem \eqref{SYSABS} has a unique solution
$$U\in \C([\tau,\infty),\H),$$
depending with continuity on the initial data.
Besides, the third component $\eta^t$
of the solution $U(t)$ fulfills
the explicit representation formula
$$
\eta^t(s)=
\begin{cases}
u(t)-u(t-s), & 0<s\le t-\tau,\\
\eta_\tau(s-t)+u(t)-u_\tau, & s>t-\tau.
\end{cases}
$$
Accordingly, for every fixed $\eps\in[0,1]$, the
map
$$S_{\eps}(t,\tau):\H\rightarrow \H,\quad t\geq \tau,$$
acting by the formula
$$
S_{\eps}(t,\tau)U_{\tau }=U(t),
$$
defines a dynamical process
on the natural weak energy space $\H$, characterized by the two properties
\begin{itemize}
\item[(i)] $S_\eps(\tau,\tau)$ is the identity map on $\H$ for every $\tau\in\R$.
\smallskip
\item[(ii)] $S_\eps(t,\tau)S_\eps(\tau,r)=S_\eps(t,r)$ for every $t\geq\tau\geq r$.
\smallskip
\end{itemize}
Moreover, the family $S_{\eps}(t,\tau)$ generated by problem~\eqref{SYSABS}
fulfills the joint continuity
\begin{itemize}
\item[(iii)] $(t,U_\tau)\mapsto S_{\eps}(t,\tau)U_\tau\in\C([\tau,\infty)\times\H,\H)$ for every $\tau\in\R$.
\end{itemize}

\section{Dissipativity}
\label{SDis}

\subsection{Uniform absorbing sets}
The first step towards the asymptotic analysis of the process $S_\eps(t,\tau)$ is
an {\it a priori} estimate on the solutions $U(t)=S_\eps(t,\tau)U_\tau$.
With $\Phi$ and $Q_\eps$ given by~\eqref{phii} and~\eqref{Q}, respectively,
the main result of this section reads as follows.

\begin{theorem}
\label{Th-EE}
For every fixed $\eps\in [0,1]$, every $t\geq\tau$ and every initial datum $U_{\tau }\in \H$,
we have the estimate
$$
\Phi(U(t))\leq \Q(\|U_\tau\|_{\H})\e^{-\omega (t-\tau )} +c\left(1+Q_\eps\right),
$$
where $\omega >0$ is a universal constant independent of $\eps$ and $\tau$.
\end{theorem}

The theorem has a straightforward corollary.

\begin{corollary}
\label{corco}
For every fixed $\eps\in [0,1]$, the process $S_{\eps}(t,\tau)$ has a uniform
(with respect to $\tau \in \R$) absorbing set.
\end{corollary}

By definition, this is
a bounded set $\B^\eps\subset\H$ with the following property: for any
bounded set $\B\subset \H$ of initial data assigned at time $\tau\in\R$,
there is an entering time $t_\e=t_\e(\B,\eps)>0$,
independent of $\tau$, such that
$$
S_{\eps}(t,\tau)\B\subset \B^{\eps },\quad \forall t\geq \tau +t_\e.
$$
It is then apparent after
Theorem \ref{Th-EE} that one can take as $\B^\eps$ the closed subset of $\H$
\begin{equation}
\B^{\eps }=\big\{U\in\H:\;\Phi(U)\leq R \big\},
\label{s1r32}
\end{equation}
for any fixed
$$R>c\left( 1+Q_\eps\right).$$

\begin{remark}
Note that, although $\B^\eps$ is bounded in $\H$ for every given $\eps$,
its norm blows
up to infinity in the limit $\eps\to 0$.
\end{remark}

The remaining of the section is devoted to the proof of Theorem \ref{Th-EE}.

\subsection{A preliminary lemma}
The main tool needed in the proof is a Gronwall-type lemma from~\cite{Patonw}.

\begin{lemma}
\label{LEMMAPatonwall}
Let $\Lambda_\oo$ be a family of absolutely continuous nonnegative functions
on $[\tau,\infty)$
satisfying for every $\oo>0$ small
the differential inequality
$$
\ddt\Lambda_\oo(t)+\oo \Lambda_\oo(t)\le c\oo^2[\Lambda_\oo(t)]^\beta+\frac{g(t)}\oo,
$$
where $1\leq \beta<\frac32$ and
$$\sup_{t\geq \tau}\int_{t}^{t+1} |g(y)|\d y=M<\infty.$$
Moreover, let $\phi$ be a continuous nonnegative function on $[\tau,\infty)$
such that
$$\frac1{C_0}\,\phi(t)\leq \Lambda_\oo(t)\leq C_0\phi(t)+C_1$$
for every $\oo>0$ small and some $C_0\geq 1$, $C_1\geq 0$.
Then, there exist $\omega>0$, $C\geq 0$ and an increasing positive function $\Q(\cdot)$
such that
$$\phi(t)\leq \Q(\phi(\tau))\e^{-\omega (t-\tau)}+CM+C.$$
If $M=0$, the constant $C$ is zero as well, yielding the exponential decay of $\phi$.
\end{lemma}

\subsection{Energy functionals}
Let now $\eps\in [0,1]$ be fixed, and let
$$U(t)=(u(t), \pt u(t), \eta^t)$$
be the solution to \eqref{PROBLEM} (or \eqref{SYSABS} which is the same)
originating from a given $U_\tau\in\H$ at time $t=\tau$.
In what follows, we will use several times without explicit mention the Young,
H\"older and Poincar\'e inequalities.
We will also perform several formal computations, all justified within a suitable
regularization scheme. The reader is addressed to~\cite{PAT},
where the same estimates have been carried out for the linear model.

\smallskip
\noindent
$\bullet$
We begin to introduce the main energy functional
$$E(t)=\frac12\|U(t)\|_{\H}^2+{\mathcal F}(u(t))+c_E.
$$
Up to choosing the constant $c_E>0$ sufficiently large,
it is clear from~\eqref{CTRLF} that
\begin{equation}
\label{Pippo1}
\frac1c\Phi(U(t))\leq E(t)\leq c\Phi(U(t))+c, \quad c\geq 1.
\end{equation}
The basic multiplication of \eqref{SYSABS} by $U$ in $\H$ gives
\begin{equation}
\label{STIMALYA1}
\ddt E+I=\l g^\eps,\pt u\r\leq \|g^\eps\|\|\pt u\|,
\end{equation}
having set
$$I(t)=-\frac12\int_0^\infty \mu'(s)\|\eta^t(s)\|_1^2\d s.$$
On account of~\eqref{K2}, we have the control
\begin{equation}
\label{stimagetta}
\frac\delta2\|\eta^t\|_\M^2\leq I(t).
\end{equation}

\smallskip
\noindent
$\bullet$
Next, in order to handle the possible singularity of $\mu$ at zero,
we choose $\varpi>0$ small, to be properly fixed later, and $s_\varpi>0$ such that
$$\int_{0}^{s_\varpi}\mu(s)\d s \leq \frac{\varpi\kappa_0}{2}.$$
Setting
$$
\mu_\varpi(s)=
\begin{cases}
\mu(s_\varpi), & 0<s\leq s_\varpi,\\
\mu(s), & s>s_\varpi,
\end{cases}
$$
we introduce the auxiliary functionals
\begin{align*}
&L_1(t)=-\frac{1}{\kappa_0}\int_0^\infty\mu_\varpi(s)\l \pt u(t),\eta^t(s)\r \d s,\\
&L_2(t)=\l \pt u(t), u(t)\r.
\end{align*}
Then, we have the inequality (cf.\ \cite{GMPZ,PAT})
\begin{align*}
\label{diff1}
\ddt L_1+(1-\varpi)\|\pt u\|^2
&\leq \epsilon_\varpi\|u\|_1^2
+c\|\eta\|_{\M }^2
+c I +\int_0^\infty\mu(s)|\l f(u),\eta(s)\r|\d s\\
&\quad+\int_0^\infty\mu(s)|\l g^\eps,\eta(s)\r|\d s
\end{align*}
for some $\epsilon_\varpi>0$, with the property that $\epsilon_\varpi\to 0$ as $\varpi\to 0$.
Here the constant $c$ may possibly blow up when $\varpi\to 0$.
Exploiting \eqref{utile}, and subsequently using \eqref{CTRLF},
\begin{align*}
\int_0^\infty\mu(s)|\l f(u),\eta(s)\r|\d s&\leq c\|f(u)\|_{L^{6/5}}\int_0^\infty\mu(s)\|\eta(s)\|_1\d s\\
&\leq c|{\mathcal F}(u)|^{\frac{p}{p+1}}\|\eta\|_{\M}+c\|\eta\|_{\M}\\
\noalign{\vskip2mm}
&\leq c\|u\|^{p}_{L^{p+1}}\|\eta\|_{\M}+c\|\eta\|_{\M}^2+c,
\end{align*}
while
$$\int_0^\infty\mu(s)|\l g^\eps,\eta(s)\r|\d s\leq c\|g^\eps\|^2+c\|\eta\|_\M^2.$$
Hence, recalling \eqref{stimagetta}, we end up with
\begin{equation}
\label{diff1}
\ddt L_1+(1-\varpi)\|\pt u\|^2
\leq \epsilon_\varpi\|u\|_1^2
+c I+c\|u\|^{p}_{L^{p+1}}\sqrt{I}\,
+c\|g^\eps\|^2+c.
\end{equation}
Concerning $L_2$, we have the equality
$$
\ddt L_2+\| u \|_1^2+\l f(u),u\r
= \|\pt u\|^2 -\l u,\eta\r_{\M}+\l g^\eps,u\r,
$$
and by means of \eqref{CTRLf} and \eqref{stimagetta} we obtain
\begin{equation}
\label{diff2}
\ddt L_2+\frac12\|u\|_1^2+d\|u\|^{p+1}_{L^{p+1}}
\leq \|\pt u\|^2 +c I+c\|g^\eps\|^2+c.
\end{equation}
At this point,
we define
$$L(t)=2L_1(t)+L_2(t),
$$
noting that
\begin{equation}
\label{Pippo2}
|L(t)|\leq c \|U(t)\|_{\H}^2.
\end{equation}
Collecting \eqref{diff1}-\eqref{diff2}, and fixing $\varpi$ suitably small,
we draw the differential inequality
\begin{equation}
\label{STIMALYA2}
\ddt L+\alpha\Big[\|u\|_1^2+\|\pt u\|^2+\|u\|^{p+1}_{L^{p+1}}\Big]
\leq
c I+c\|u\|^{p}_{L^{p+1}}\sqrt{I}\,
+c\|g^\eps\|^2+c,
\end{equation}
where, say, $\alpha=\min\{\frac14,d\}$.

\subsection{Proof of Theorem \ref{Th-EE}}
We introduce the family of energy functionals depending on $\oo>0$ small
$$\Lambda_\oo(t)=E(t)+\oo L(t).$$
Due to \eqref{Pippo1} and \eqref{Pippo2}, for every $\oo$ small enough
we have the control
\begin{equation}
\label{nore}
\frac1c\Phi(U(t))\leq \Lambda_\oo(t)\leq c\Phi(U(t))+c, \quad c\geq 1.
\end{equation}
Besides, from \eqref{STIMALYA1} and \eqref{STIMALYA2} we deduce
the family of differential inequalities
\begin{align*}
&\ddt\Lambda_\oo+\alpha\oo\Big[\|u\|_1^2+\|\pt u\|^2+\|u\|^{p+1}_{L^{p+1}}\Big]+(1-c\oo)I\\
&\leq c\oo\|u\|^{p}_{L^{p+1}}\sqrt{I}\,
+c\|g^\eps\|^2+\|g^\eps\|\|\pt u\|+c,
\end{align*}
which, after simple manipulations and a further use of \eqref{stimagetta}, enhances to
$$\ddt\Lambda_\oo+\alpha\oo\Phi(U)
\leq c\oo^2\|u\|^{2p}_{L^{p+1}}
+\frac{c}\oo\|g^\eps\|^2+c,
$$
for all $\oo>0$ sufficiently small.
Observing that
$$\|u\|^{2p}_{L^{p+1}}\leq [\Phi(U)]^{\frac{2p}{p+1}},
$$
and using the double control provided by \eqref{nore},
we finally obtain (up to redefining $\oo$),
$$\ddt\Lambda_\oo+\oo\Lambda_\oo
\leq c\oo^2\Lambda_\oo^\beta
+\frac{c}\oo\|g^\eps\|^2+c,
$$
with
$$\beta=\frac{2p}{p+1}.$$
Note that $\beta\in[1,\frac32)$, since $p\in[1,3)$.
Thus, having in mind~\eqref{s1BOUND}, the latter inequality
together with \eqref{nore}
allow us to apply Lemma~\ref{LEMMAPatonwall} with $\phi(t)=\Phi(U(t))$,
yielding the desired conclusion.
This finishes the proof of Theorem~\ref{Th-EE}.
\qed

\section{Uniform Global Attractors}
\label{SUni}

\subsection{Translation compact external forces}
We make the following assumption:
\begin{equation}
\label{TrCPT}
\text{both } g_0\text{ and } g_1\text{ are translation compact in } L_{\rm loc}^2(\R;\HH).
\end{equation}
By definition, this means
that for $\imath=0,1$ the set of translates
$$\textsf{T}(g_\imath)=\{g_{\imath}(\cdot+y):\,y\in \R\}
$$
is precompact in $L_{\rm loc}^2(\R;\HH)$.

\begin{definition}
The closure of the set $\textsf{T}(g_\imath)$ in the space $L_{\rm loc}^2(\R;\HH)$
is called the {\it hull} of $g_\imath$, and is denoted by $\textsf{H}(g_\imath)$.
\end{definition}

\begin{remark}
Given any $g$ translation compact in $L_{\rm loc}^2(\R;\HH)$,
then a function $\hat g$ belongs to $\textsf{H}(g)$ if and only if
there exists a
sequence $y_n\in\R$ such that
$$\lim_{n\to\infty}\,\int_a^b\|g(t+y_n)-\hat g(t)\|^2\d t=0,\quad\forall b>a.
$$
We address the reader to \cite{CVbook} for more details on translation compact functions.
\end{remark}

It is easily seen that $g^\eps$
is translation compact in $L_{\rm loc}^2(\R;\HH)$ if and only if~\eqref{TrCPT} holds.
In that case, a function $\hat g^\eps$ belongs to the hull $\textsf{H}(g^\eps)$ of $g^\eps$
if and only if
$$\hat g^\eps (t)=\hat g_{0}(t)+\eps ^{-\rho }\hat g_{1}(t/\eps),
$$
for some $\hat g_\imath\in \textsf{H}(g_\imath)$.
It is also apparent that
$$\| \hat g^\eps\|_{{\rm tb}}\leq \|g^\eps\|_{{\rm tb}}.$$
Thus, on account of~\eqref{s1BOUND}, we obtain the bound
$$
\| \hat g^\eps\|_{{\rm tb}}^2
\leq Q_\eps,\quad\forall \hat g^\eps\in \textsf{H}(g^\eps).
$$

\subsection{The family of processes}
We now consider, in place of the single problem~\eqref{SYSABS}, the family of equations
\begin{equation}
\label{hatPROBLEM}
\begin{cases}
\displaystyle
\ddt U(t)={\mathbb A} U(t)+\hat{\mathbb F}_\eps (U(t),t),\\
\noalign{\vskip1.5mm}
U(\tau)=U_\tau,
\end{cases}
\end{equation}
where
$$\hat{\mathbb F}_\eps(U,t)=\big(0,\hat g^\eps(t)-f(u),0\big),\quad \hat g^\eps \in\textsf{H}(g^\eps).
$$
Clearly, for any given $\hat g^\eps$, problem~\eqref{hatPROBLEM} generates a
(jointly continuous) dynamical process
$$S_{\hat g^\eps}(t,\tau):\H\to\H.$$
With no changes in the proof, the analogue of Theorem~\ref{Th-EE} holds. Namely, for every
$t\geq\tau$ and every initial datum $U_{\tau }\in \H$, the solution
$$U(t)=S_{\hat g^\eps}(t,\tau)U_\tau$$
fulfills the estimate
$$
\Phi(U(t))\leq \Q(\|U_\tau\|_{\H})\e^{-\omega (t-\tau )} +c\left(1+Q_\eps\right).
$$
In particular, arguing as in Corollary~\ref{corco},
the family $S_{\hat g^\eps}(t,\tau)U_\tau$ possesses an absorbing set (that we keep calling $\B^\eps$),
which is uniform with respect to both $\tau \in \R$ and $\hat g^\eps \in\textsf{H}(g^\eps)$.

\subsection{Existence of uniform global attractors}
We begin with two definitions.

\begin{definition}
A set ${\mathcal K}\subset \H$ is said to be
{\it uniformly} (with respect to $\tau\in\R$) {\it attracting} for the process
$S_\eps(t,\tau )$ if for any bounded set $\B\subset \H$ we have
the limit relation
$$
\lim_{t-\tau \to \infty }\,
\mathrm{dist}_{\H}\big(S_\eps(t,\tau)\B,{\mathcal K}\big)=0.
$$
\end{definition}

\begin{definition}
A compact set $\A^\eps\subset \H$ is said to be the
{\it uniform global attractor} of the process
$S_\eps(t,\tau )$ if it
is the minimal uniformly attracting set. The minimality property means that $\A^\eps$ belongs to
any compact uniformly attracting set.
\end{definition}

The following holds.

\begin{theorem}
\label{s1pro1}
Let \eqref{TrCPT} hold. Then, for any fixed $\rho\in [0,1]$
and $\eps\in[0,1]$, the process $S_{\eps}(t,\tau):\H\to\H$ possesses
the uniform global attractor $\A^\eps$.
\end{theorem}

\begin{remark}
\label{miii}
What one actually proves is the existence of the uniform global attractor $\A_{\textsf{H}(g^\eps)}$ for the
family of processes generated by~\eqref{hatPROBLEM}
$$\big\{ S_{\hat g^\eps}(t,\tau):\,\hat g^\eps\in \textsf{H}(g^\eps)\big\}.$$
Such an object satisfies the stronger attraction property
$$
\lim_{t-\tau \to \infty }\Big[\sup_{\hat g^\eps\in \textsf{H}(g^\eps)}\,
\mathrm{dist}_{\H}\big(S_{\hat g^\eps}(t,\tau)\B,\A_{\textsf{H}(g^\eps)}\big)\Big]=0.
$$
However, since it is standard matter to prove the continuity of the map
$$\hat g^\eps\mapsto S_{\hat g^\eps}(t,\tau) U_\tau: \textsf{H}(g^\eps)\to \H,$$
for every fixed $t\geq \tau$ and $U_\tau\in \H$, we draw from \cite[Theorem 29]{CCP} the equality
$$\A_{\textsf{H}(g^\eps)}=\A^\eps.$$
\end{remark}

The proof of the theorem exploits the next abstract result from \cite{CV1a,CV1,CVbook} (see also \cite{CCP}).

\begin{theorem}
\label{absatt}
Assume that the process is asymptotically compact, that is, there exists
a compact uniformly attracting set. Then there is the (unique) uniform global attractor.
\end{theorem}

A way to prove asymptotic compactness, in fact of the whole family of processes
$S_{\hat g^\eps}(t,\tau)$, is to show that
\begin{equation}
\label{kur}
\lim_{t-\tau\to \infty }\Big[\sup_{\hat g^\eps\in \textsf{H}(g^\eps)}\,
\alpha_\H\big(S_{\hat g^\eps}(t,\tau)\B^\eps\big)\Big]=0,
\end{equation}
where $\B^\eps$ is a uniform absorbing set for $S_{\hat g^\eps}(t,\tau)$,
and
$$\alpha_\H(\B)=\inf\big\{d: \text{$\B$ has a finite cover of balls of
$\H$ of diameter less than $d$}\big\}$$
denotes the
{\it Kuratowski measure of noncompactness}
of a bounded
set $\B\subset \H$ (see \cite{HAL} for more details on $\alpha_\H$).

\begin{proof}[Proof of Theorem \ref{s1pro1}]
In order to verify \eqref{kur}, for an arbitrarily fixed $\tau\in\R$, $U_\tau\in\B^\eps$
and $\hat g^\eps\in\textsf{H}(g^\eps)$,
let us decompose the solution $U(t)$ to~\eqref{hatPROBLEM} into the sum
$$U(t)=V(t)+W(t),$$
where
$$
\begin{cases}
\displaystyle
\ddt V(t)={\mathbb A} V(t),\\
\noalign{\vskip1.5mm}
V(\tau)=U_\tau,
\end{cases}
$$
and
$$
\begin{cases}
\displaystyle
\ddt W(t)={\mathbb A} W(t)+\hat{\mathbb F}_\eps (U(t),t),\\
\noalign{\vskip1.5mm}
W(\tau)=0.
\end{cases}
$$
The solution $V(t)$ to the first (linear) autonomous problem
can be written as
$$V(t)=S(t-\tau)U_\tau,$$
where $S(t)$ is an exponentially stable (contraction) semigroup on $\H$ (see \cite{PAT}).
Accordingly,
$$\|V(t)\|_\H^2\leq C\e^{-\omega (t-\tau )},$$
for some constant $C>0$ depending only on $\B^\eps$.
Such a conclusion can also be drawn from Theorem~\ref{Th-EE}.
Concerning $W(t)$, via the Duhamel representation formula we have
$$W(t)=W(t;\tau,\hat g^\eps)=\int_\tau^t S(t-y) \hat{\mathbb F}_\eps (U(y),y)\d y.
$$
Consequently, \eqref{kur} follows if one proves the precompactness in $\H$ of the set
$$
{\mathcal K}_{t,\tau}=\bigcup_{U_\tau\in\B^\eps}
\bigcup_{\hat g^\eps\in \textsf{H} (g^\eps)}
W(t;\tau,\hat g^\eps)
$$
for every fixed $t\geq\tau$. This can be done, with no essential differences, as in the case
of the nonautonomous damped hyperbolic equation treated in detail in \cite{CVbook} (see the proof of Proposition~VI.4.3 therein).
\end{proof}

Since the attractor is contained in any closed uniform absorbing set,
we learn from~\eqref{s1r32} that
\begin{equation}
\label{sizea}
\Phi(\A^\eps)\leq \frac{Q}{\eps ^{2\rho}},\quad\forall \eps\in(0,1],
\end{equation}
for some $Q>0$ independent of $\eps$.
In turn, this gives the bound
$$
\| \A^\eps\| _{\H}\leq \frac{c}{\eps ^{\rho}}.
$$
Thus, in principle, the size of the global attractor $\A^{\eps}$ of
equation~\eqref{PROBLEM} with singularly oscillating terms can grow to infinity
as the oscillating rate $1/\eps \to \infty$.

\subsection{Structure of the attractors}
We now provide a complete description
of the structure of the global attractors $\A^{\eps }$.

\begin{definition}
Let $\hat g^\eps\in\textsf{H}(g^\eps)$ be fixed. A function $y\mapsto U(y):\R\to \H$ is called
a {\it complete bounded trajectory} ({\cbt})
of $S_{\hat g^\eps}(t,\tau)$ if
\begin{itemize}
\item[(i)] $\sup_{y\in\R}\|U(y)\|_\H<\infty$, and
\smallskip
\item[(ii)] $U(y)=S_{\hat g^\eps}(y,\tau)U(\tau)$ for every $y\geq \tau$ and every $\tau\in\R$.
\end{itemize}
The {\it kernel section} of $\hat g^\eps$ at time $y$ is the set
$${\mathbb K}_{\hat g^\eps}(y)=\big\{U(y):\,U \text{ is a {\cbt} of }S_{\hat g^\eps}(t,\tau)\big\}.
$$
\end{definition}

We have now all the ingredients to state our characterization theorem,
which follows from the results of the recent paper~\cite{CCP},
generalizing the theory presented in~\cite{CVbook}.

\begin{theorem}
\label{THMChar}
Let \eqref{TrCPT} hold. Then,  for every $\eps\in[0,1]$,
the global attractor $\A^{\eps}$ of the process $S_\eps(t,\tau)$
has the form
$$
\A^{\eps }=\bigcup_{\hat g^\eps\in {\text{\rm \textsf{H}}} (g^\eps)}
{\mathbb K}_{\hat g^\eps}(y),
$$
for an arbitrarily given $y\in\R$. Moreover,
${\mathbb K}_{\hat g^\eps}(y)$ is non empty
for every $y\in\R$ and $\hat g^\eps\in \text{\rm \textsf{H}}(g^\eps)$.
\end{theorem}

\begin{remark}
Indeed, according to \cite{CCP}, it is enough to prove that the map
$$(U_0,\hat g^\eps)\mapsto S_{\hat g^\eps}(y_\star,0)U_0:\H\times \textsf{H}(g^\eps)\to \H$$
is closed for some $y_\star>0$.\footnote{Recall that a map $\psi:X\to Y$ is
closed if $\psi(x)=y$ whenever $x_n\to x$ and $\psi(x_n)\to x$.}
\end{remark}

\section{An Auxiliary Linear Problem}
\label{SAux}

\noindent
For further scopes,
we now consider for $\eps>0$ the family of auxiliary problems
on $[\tau,\infty)$
\begin{equation}
\label{LINEAReps}
\begin{cases}
\partial_{tt} v+Av
+\displaystyle\int_0^\infty \mu(s) A\zeta(s)\d s= k(t/\eps),\\
\zeta_t=T\zeta+\pt v,\\
\noalign{\vskip1.5mm}
(v(\tau),\pt v(\tau),\zeta^\tau)=(0,0,0),
\end{cases}
\end{equation}
where $k\in L^{2}_{\rm loc}(\R;\HH^{\sigma})$ for some $\sigma\in\R$.
Setting
$$K(t,\tau)=\int_{\tau}^{t}k(y)\d y,\quad t\geq\tau,$$
the following holds.

\begin{proposition}
\label{propLINEAR}
Assume that
\begin{equation}
\sup_{t\geq\tau,\,\tau\in\R}
\left\{\|K(t,\tau)\|_{\sigma-1}^2+\int_{t}^{t+1}\|K(y,\tau)\|_{\sigma }^2\d y\right\}\leq \ell^2,
\label{s2r9}
\end{equation}
for some $\ell\geq 0$.
Then,  problem
\eqref{LINEAReps} has a unique solution
$V(t)=(v(t),\partial _{t}v(t),\zeta^t)$
satisfying
$$
\|V(t)\|_{\H^{\sigma-1}}\leq c \ell\eps,\quad \forall t\geq \tau,
$$
where $c>0$ is independent of $k$.
\end{proposition}

\begin{remark}
Condition \eqref{s2r9} is satisfied, for instance,
if $k\in
L^\infty(\R;\HH^{\sigma-1})\cap L^1_{\rm loc}(\R;\HH^{\sigma})$
is a time periodic function of period $\Pi>0$
having zero mean, i.e
$$\int_{0}^{\Pi}k(y) \d y=0.$$
Other examples of quasiperiodic
and almost periodic in time functions satisfying \eqref{s2r9} can be found
in \cite{CV1a,CVbook}.
\end{remark}

The proof of the proposition requires a preliminary lemma.

\begin{lemma}
\label{s2lem1}
The unique solution $V(t)=(v(t),\partial _{t}v(t),\zeta^t)$ to the problem~\eqref{LINEAReps}
with $\eps=1$
fulfills the inequality
$$
\|V(t)\|_{\H^\sigma}^2\leq c\int_{\tau}^{t}{\e}^{-\omega (t-y)}\|k(y)\|_{\sigma }^2\d y,
$$
for every $t\geq\tau$ and some $\omega >0$
independent of the initial time $\tau$.
\end{lemma}

\begin{proof}
Existence and uniqueness of the linear problem follows by standard semigroup arguments,
which are applicable in any space $\H^\sigma$ (see e.g.\ \cite{PAT}).
Arguing as in the proof of Theorem \ref{Th-EE},
and using the fact that here $f\equiv 0$, it is not difficult to prove the differential inequality
$$
\ddt \|V\| _{\H^\sigma }^{2}+\omega \|V\| _{\H^\sigma}^{2}\leq c\|k\|_{\sigma }^2,
$$
for some $\omega>0$.
The desired result follows by the (classical) Gronwall lemma.
\end{proof}

\begin{proof}[Proof of Proposition \ref{propLINEAR}]
Without loss of generality, we may assume $\tau=0$.
Denoting
$$
\tilde v(t)=\int_{0}^{t}v(y)\d y,\qquad \tilde \zeta^t(s)=
\int_{0}^{t}\zeta^y(s)\d y,
$$
an integration of~\eqref{LINEAReps} in time yields
$$
\begin{cases}
\partial_{tt} \tilde v+A\tilde v
+\displaystyle\int_0^\infty \mu(s) A\tilde \zeta(s)\d s= K_\eps(t),\\
\tilde \zeta_t=T\tilde \zeta+\pt \tilde v,\\
\noalign{\vskip1.5mm}
(\tilde v(0),\pt \tilde v(0),\tilde \zeta^0)=(0,0,0),
\end{cases}
$$
where
$$K_{\eps}(t)=\int_{0}^{t}k\left(y/\eps\right) \d y=\eps K\left( t/\eps,0\right).$$
Then, we easily infer from~\eqref{s2r9} that
$$
\sup_{t\geq 0}
\left\{\|K_{\eps}(t)\|_{\sigma-1}^2+\int_{t}^{t+1}\|K_{\eps}(y)\|_{\sigma }^2\d y\right\}\leq c\ell^2\eps^2.
$$
Applying Lemma~\ref{s2lem1} to the system above we obtain
$$
\|\tilde v(t)\|_{\sigma +1}^2+\|\partial _{t}\tilde v(t)\|_{\sigma }^2+\|\tilde \zeta^t\|_{\M^\sigma}^2\leq
c\int_{0}^{t}\e^{-\omega (t-y)}\|K_{\eps }(y)\|_{\sigma}^2\d y
\leq c\ell^2\eps^2,
$$
where the last passage follows from the well-known inequality (see e.g.\ \cite{PPV})
$$
\sup_{t\geq 0}\int_0^t \e^{-\omega(t-s)}h(y)\d s
\leq  \frac{1}{1- \e^{-\omega}}\, \sup_{t\geq0}\int_t^{t+1}
h(y)\d y,
$$
valid for every nonnegative
locally summable function $h$
and every $\omega>0$.
In particular, we learn that
$$
\|v(t)\|_{\sigma }=\|\partial _{t}\tilde v(t)\|_{\sigma }\leq
c \ell\eps.
$$
Besides, by comparison in the equation
$$
\|\partial_{tt}\tilde v(t)\|_{\sigma-1}
\leq \|A \tilde v(t)\|_{\sigma-1}
+\int_0^\infty\mu(s)\|A\tilde \zeta^t(s)\|_{\sigma-1}\d s +\|K_{\eps }(t)\|_{\sigma -1}.
$$
But
$$
\|A \tilde v(t)\|_{\sigma -1}=\|\tilde v(t)\|_{\sigma +1}\leq c \ell\eps,$$
and
$$
\int_0^\infty\mu(s)\|A\tilde \zeta^t(s)\|_{\sigma-1}\d s
\leq c\|\tilde \zeta^t\|_{\M^\sigma}\leq c\ell\eps.
$$
Accordingly,
$$\|\pt v(t)\|_{\sigma-1}
=\|\partial_{tt}\tilde v(t)\|_{\sigma-1}\leq c \ell\eps.
$$
We are left to prove the estimate
$$\|\zeta^t\|_{\M^{\sigma-1}}\leq c \ell\eps.$$
To this end,
we multiply the second equation of the system
by $\zeta^t$ in $\M^{\sigma-1}$. Making use of \eqref{K2}, we get
$$\ddt \|\zeta^t\|^2_{\M^{\sigma-1}}+\delta \|\zeta^t\|^2_{\M^{\sigma-1}}\leq c\|\pt v(t)\|_{\sigma-1}^2\leq c\ell^2\eps^2,$$
and the claim follows from the Gronwall lemma.
\end{proof}

\section{Uniform Boundedness of the Global Attractors}
\label{SBdd}

\subsection{Statement of the result}
Estimate \eqref{sizea} provides a bound on the size of the attractors which, unless $\rho=0$, is not uniform
with respect to $\eps \in [0,1]$.
Here, as far as the more interesting case $\rho>0$ is concerned, we give
a sufficient condition in order for the family $\A^{\eps }$ to be uniformly bounded in $\H$.
Such a condition involves only the function $g_{1}$, which introduces singular oscillations in the external force.
Setting
\begin{equation*}
G_{1}(t,\tau)=\int_{\tau}^{t}g_{1}(y)\d y,\quad t\geq\tau,
\end{equation*}%
our main assumption reads
\begin{equation}
\sup_{t\geq\tau,\,\tau\in\R}
\left\{\|G_1(t,\tau)\|_{\vartheta-1}^2+\int_{t}^{t+1}\|G_1(y,\tau)\|_{\vartheta}^2\d y\right\}\leq \ell^2,
\label{s3r0}
\end{equation}%
for some $\ell\geq 0$,
where
$$
\vartheta=
\begin{cases}
1& \text{if}\quad 1\leq p\leq 2,\\
\noalign{\vskip1mm}
\displaystyle \frac{3(p-1)}{p+1}& \text{if}\quad 2<p<3.
\end{cases}
$$

\begin{theorem}
\label{s3the1}
Let \eqref{TrCPT} hold, and let $G_{1}$ satisfy \eqref{s3r0}.
Then, for every $\rho\in [0,1)$, the global attractors $\A^{\eps }$ are uniformly bounded in $\H$, i.e.\
$$
\sup_{\eps\in[0,1]}\| \A^{\eps }\| _{\H}<\infty.
$$
\end{theorem}

Actually, as it will be clear in the upcoming proof, in the Lipschitz situation $p=1$
the result extends to the limiting case $\rho=1$ as well.

\subsection{Proof of Theorem \ref{s3the1}}
Till the end of the section, $\rho\in(0,1)$ and $p\in[1,3)$ are understood to be fixed.
The key argument is in the following lemma.

\begin{lemma}
\label{key}
Let $Q>0$ and $\gamma>0$ be given constants.
Assume to have the bound
$$\Phi(\A^\eps)\leq \frac{Q}{\eps^{2\gamma}},\quad\forall \eps\in (0,1].$$
Then, there exists $\widehat Q>0$ such that
$$\Phi(\A^\eps)\leq \frac{\widehat Q}{\eps^{2\gamma m_\gamma}},$$
where
$$m_\gamma=\max\bigg\{0,\frac{2(p-1)}{p+1}-\frac{1-\rho}{\gamma}\bigg\}.$$
In particular, if
$m_\gamma=0$, the family $\A^\eps$ is uniformly bounded in $\H$.
\end{lemma}

\begin{remark}
\label{remkey}
Note that if $p=1$ then $m_\gamma=0$ for every $\gamma>0$
(even if $\rho=1$). Instead, if $p>1$, the conclusion
$m_\gamma=0$ holds whenever
\begin{equation}
\label{gammastar}
\gamma\leq \gamma_\star:=\frac{(1-\rho)(p+1)}{2(p-1)}.
\end{equation}
\end{remark}

\begin{proof}[Proof of Lemma \ref{key}]
Let $\eps\in(0,1]$ be fixed, and let $U(t)=(u(t),\pt u(t),\eta^t)$ be any {\cbt} lying on the attractor $\A^\eps$.
Thus, $U(t)$ solves for all times problem~\eqref{hatPROBLEM} for some
$$\hat g^\eps(t)=\hat g_{0}(t)+\eps ^{-\rho }\hat g_{1}(t/\eps)\in \textsf{H}(g^\eps).
$$
In particular (see e.g.\ \cite{CPV-aver}), the function
$$
\hat G_{1}(t,\tau)=\int_{\tau}^{t}\hat g_{1}(y)\d y,\quad t\geq\tau,
$$
fulfills the analogue of~\eqref{s3r0}.
In the light of the assumptions, the characterization
Theorem~\ref{THMChar} implies that
\begin{equation}
\label{unib}
\Phi(U(t))\leq \frac{Q}{\eps^{2\gamma}},\quad\forall t\in\R.
\end{equation}
We divide the proof in a number of steps.
In what follows, $\tau\in\R$ will be an arbitrary initial time.

\subsection*{Step I}
For $t>\tau$, let $V(t)=(v(t),\partial _{t}v(t),\zeta^t)$ be the solution
to the auxiliary problem
\begin{equation}
\label{VVV}
\begin{cases}
\partial_{tt} v+Av
+\displaystyle\int_0^\infty \mu(s) A\zeta(s)\d s= \eps^{-\rho}\hat g_1(t/\eps),\\
\zeta_t=T\zeta+\pt v,
\end{cases}
\end{equation}
with null initial datum
$$V(\tau)=0.$$
On account of Proposition~\ref{propLINEAR}, we have
the inequality
\begin{equation}
\label{vpicpic}
\|V(t)\|_{\H^{\vartheta-1}}\leq c\ell \eps ^{1-\rho}.
\end{equation}
Then, from the Sobolev embedding theorem
$$\HH^{\vartheta}\subset \HH^{3(p-1)/(p+1)}\subset L^{2(p+1)/(3-p)}(\Omega ),
$$
we deduce the estimate
\begin{equation}
\label{vpic}
\| v\|_{L^{2(p+1)/(3-p)}}\leq c \| v\|_\vartheta\leq c\ell \eps^{1-\rho}.
\end{equation}

\subsection*{Step II} The difference
$$W(t)=(w(t),\pt w(t),\xi^t)=U(t)-V(t)$$
fulfills the system
$$
\begin{cases}
\partial_{tt} w+Aw
+\displaystyle\int_0^\infty \mu(s) A\xi(s)\d s +f(w)=-[
f(w+v)-f(w)] +\hat g_{0}(t),\\
\xi_t=T\xi+\pt w,
\end{cases}
$$
with initial condition
$$
W(\tau)=U(\tau).$$
Then, we define the family of functionals $\Lambda_\oo$
as in the proof of Theorem~\ref{Th-EE}, but using now $W(t)$ in place of $U(t)$.
Recasting {\it verbatim} that proof, we draw
the analogue of~\eqref{nore}, i.e.\
\begin{equation}
\label{nore2}
\frac1c\Phi(W(t))\leq
\Lambda_\oo(t)\leq
c\Phi(W(t))+c,\quad c\geq 1,
\end{equation}
along with the family of differential inequalities
$$\ddt\Lambda_\oo+\oo\Lambda_\oo\leq
c\oo^2\Lambda_\oo^{\frac{2p}{p+1}}+\frac{c}{\oo}\|f(w+v)-f(w)\|^2+\frac{c}{\oo}\|\hat g_0\|^2+c,$$
for all $\oo>0$ small.

\subsection*{Step III}
We now estimate the term in the right-hand side above. From \eqref{crescita},
$$
|f(w+v)-f(w)|\leq c\left(1+|w|^{p-1}+|v|^{p-1}\right)|v|.
$$
Therefore, we have
the control
$$
\|f(w+v)-f(w)\|^{2}
\leq c\int_{\Omega }|w(x,\cdot )|^{2(p-1)}|v(x,\cdot
)|^{2}\d x+ c\| v\| _{L^{2p}}^{2p}+c\|v\|^2.
$$
Since $0<2(p-1)<p+1$, setting
$$
p_1=\frac{p+1}{2(p-1)}>1,\qquad p_{2}=\frac{p+1}{3-p}>1,
$$
we infer from the H\"{o}lder inequality with exponents
$(p_{1},p_{2})$ together with~\eqref{vpic} that
$$\int_{\Omega }|w(x,\cdot )|^{2(p-1)}|v(x,\cdot )|^{2}\d x \leq \|
w\| _{L^{p+1}}^{2(p-1)}\| v\| _{L^{2(p+1)/(3-p)}}^{2}
\leq c \eps^{2(1-\rho)}\| w\| _{L^{p+1}}^{2(p-1)}.
$$
As $2p<2(p+1)/(3-p)$, by a further use of \eqref{vpic} we also get
$$
\| v\| _{L^{2p}}^{2p}+\|v\|^2
\leq c\left(\| v\| _{L^{{2(p+1)/(3-p)}}}^{2p}+\| v\| _{L^{{2(p+1)/(3-p)}}}^{2}\right)\leq c.
$$
Summarizing,
$$
\|f(w+v)-f(w)\|^{2}\leq c \eps^{2(1-\rho)}\| w\| _{L^{p+1}}^{2(p-1)}
+c.
$$
Hence, denoting
$$h_\eps(t)=\eps^{2(1-\rho)}\| w(t)\| _{L^{p+1}}^{2(p-1)},
$$
we end up with
$$
\ddt\Lambda_\oo+\oo\Lambda_\oo\leq
c\oo^2\Lambda_\oo^{\frac{2p}{p+1}}+\frac{c}{\oo}h_\eps+\frac{c}{\oo}\|\hat g_0\|^2+\frac{c}{\oo}.
$$

\subsection*{Step IV} In the light of \eqref{unib} and \eqref{vpic}, it is clear that
$$\|w(t)\|_{L^{p+1}}^{p+1}\leq c\eps^{-2\gamma}.$$
In turn,
$$h_\eps\leq
c\big[\eps^{-2\gamma}\big]^{\frac{2(p-1)}{p+1}-\frac{1-\rho}{\gamma}}.$$
Thus, by the very definition of $m_\gamma$,
$$h_\eps(t)\leq c \eps^{-2\gamma m_\gamma}.$$
Accordingly, we arrive at
$$\ddt\Lambda_\oo+\oo\Lambda_\oo\leq
c\oo^2\Lambda_\oo^{\frac{2p}{p+1}}+\frac{c}{\oo}\eps^{-2\gamma m_\gamma}+\frac{c}{\oo}\|\hat g_0\|^2.$$
On account of \eqref{nore2},
an application of the Gronwall Lemma \ref{LEMMAPatonwall} yields
$$\Phi(W(t))\leq c_\eps\e^{-\omega(t-\tau)}+ c\eps^{-2\gamma m_\gamma}$$
for some $\omega>0$ and every $t\geq \tau$, where $c_\eps>0$ is a constant depending only on the size of the attractor $\A^\eps$.
It is worth noting that neither $c_\eps$ nor $c$ depend on the chosen initial time $\tau$.
Letting $\tau\to -\infty$, we finally obtain the uniform-in-time estimate
$$\Phi(W(t))\leq c \eps^{-2\gamma m_\gamma}.$$
Since $\Phi(V(t))\leq c$,
we get by comparison
$$\Phi(U(t))\leq c \eps^{-2\gamma m_\gamma}.$$
Recalling that $U(t)$ is an arbitrary {\cbt}, we are done.
\end{proof}

\begin{proof}[Conclusion of the proof of Theorem \ref{s3the1}]
Since $\A^0$ is bounded in $\H$, let us restrict our attention to the case $\eps>0$.
We know from \eqref{sizea} that
$$\Phi(\A^\eps)\leq \frac{Q}{\eps^{2\rho}},\quad\forall \eps\in (0,1].$$
Then, by an application of Lemma \ref{key} with $\gamma=\rho$, the thesis is trivially true whenever $\rho\leq \gamma_\star$,
which is the same as saying that
$$\rho\leq \rho_\star:=\frac{p+1}{3p-1}.$$
This concludes the proof when $p=1$, where $\rho_\star=1$. Note that,
in this case, the result holds also for $\rho=1$.

\smallskip
\noindent
Conversely, for $p>1$, we have to discuss those values of $\rho$ such that
$\rho_\star<\rho$.
Define
$$
\varkappa=\frac{2(p-1)}{p+1}-\frac{1-\rho}{\rho}.
$$
Note that
$$0<\varkappa<1$$
for every $p\in (1,3)$ and $\rho\in (\rho_\star,1)$.
Indeed, since $\rho>\rho_\star$, we have
$$\varkappa>\frac{2(p-1)}{p+1}-\frac{1}{\rho_\star}+1=0.$$
On the other hand,
$$\varkappa<1\quad\Leftrightarrow\quad
\rho<\frac{p+1}{2(p-1)},$$
the latter being true for every $\rho\in [0,1)$, as $p<3$ implies that the right-hand side is greater than 1.
We now prove by induction that for every $n\in\N$ there exists $Q_n>0$ such that
$$
\Phi(\A^\eps)\leq \frac{Q_n}{\eps^{2\rho\varkappa^n}},\quad\forall\eps\in(0,1].
$$
The case $n=0$ is already known. Hence, it is enough showing the implication
$$\Phi(\A^\eps)\leq \frac{Q_n}{\eps^{2\rho\varkappa^n}}
\quad \Rightarrow\quad
\Phi(\A^\eps)\leq \frac{Q_{n+1}}{\eps^{2\rho\varkappa^{n+1}}}.$$
Indeed, Lemma \ref{key} with $\gamma=\rho\varkappa^n$ yields
$$\Phi(\A^\eps)\leq \frac{{\widehat Q}_n}{\eps^{2\rho{\varkappa^{n}} m_\gamma}}, $$
with
$$m_\gamma=\max\bigg\{0,\frac{2(p-1)}{p+1}-\frac{1-\rho}{\rho\varkappa^n}\bigg\}.$$
On the other hand, since $0<\varkappa<1$, it is apparent that
$$\frac{2(p-1)}{p+1}-\frac{1-\rho}{\rho\varkappa^n}<\varkappa.
$$
Hence
$$\rho{\varkappa^{n}} m_\gamma< \rho\varkappa^{n+1},$$
and the inductive claim follows by setting $Q_{n+1}={\widehat Q}_n$.
At this point, since
$$\lim_{n\to\infty} \rho\varkappa^n= 0,$$
up to choosing $n$ large enough such that
$$\rho\varkappa^n\leq \gamma_\star,$$
an application of Lemma \ref{key} with $\gamma=\rho\varkappa^n$ and $Q=Q_n$ provides
the desired uniform bound.
\end{proof}

\section{Convergence of the Global Attractors}
\label{SCon}

\noindent
We finally establish the upper semicontinuity of the attractors as $\eps\to 0$.

\begin{theorem}
\label{s4the1}
Let \eqref{TrCPT} hold, and let $G_{1}$ satisfy \eqref{s3r0}.
Then, for every $\rho\in [0,1)$, the global attractors
$\A^{\eps }$ converge
to $\A^{0}$ with respect to the Hausdorff semidistance in
$\H$ as $\eps \to0$, i.e.\
$$
\lim_{\eps\to 0}\,
\mathrm{dist}_{\H}\big(\A^{\eps },\A^{0}\big)=0.
$$
\end{theorem}

In order to prove the theorem, we make a comparison between some particular
solutions to~\eqref{hatPROBLEM} with $\eps>0$ and those to~\eqref{hatPROBLEM} with $\eps=0$,
sharing the same value at a given time $\tau\in\R$.
For a given $\eps>0$, let us take any {\cbt}
$$U_\eps(t)=(u_\eps(t),\partial_t {u}_\eps(t), {\eta}_\eps^t)$$
of the process $S_{\hat g^\eps}(t,\tau)$, for some
$$\hat g^\eps(t)=\hat g_{0}(t)+\eps ^{-\rho }\hat g_{1}(t/\eps)\in \textsf{H}(g^\eps).$$
Then, for an arbitrarily fixed $\tau\in\R$, consider the solution (corresponding to $\eps=0$)
$$U_0(t)=S_{\hat g^0}(t,\tau)U_\eps(\tau)=(u_0(t),\partial_t {u}_0(t), {\eta}_0^t).$$
Due to Theorem~\ref{s3the1}, along with the estimate of Theorem~\ref{Th-EE} to handle the case $\eps=0$, we have the uniform bound
\begin{equation}
\label{s4r9a}
\sup_{\eps\in[0,1]}\,\|U_{\eps}(t)\|_{\H}\leq c,\quad\forall t\geq \tau\in\R.
\end{equation}
Next, we define the deviation
$$
\bar{U}(t)=U_{\eps}(t)-U_{0}(t)=(\bar{u}(t),\partial_t \bar{u}(t), \bar{\eta}^t).
$$

\begin{lemma}
\label{s4pro1}
We have the estimate
$$
\|\bar{U} (t)\| _{\H}
\leq c\eps ^{1-\rho}\e^{c(t-\tau )},\quad\forall
t\geq \tau.
$$
Here, $c>0$ is independent of $\eps$, $\tau$, $\hat g^\eps$ and of the choice of $U_\eps(t)$.
\end{lemma}

\begin{proof}
Let $V(t)=(v(t),\partial _{t}v(t),\zeta^t)$ be the solution to the auxiliary problem~\eqref{VVV} with null initial datum
$V(\tau)=0$. The difference
$$W(t)=\bar{U}(t)-V(t)=(w(t),\partial_t w(t), \xi^t)$$
fulfills the problem
$$
\begin{cases}
\partial_{tt} w+Aw
+\displaystyle\int_0^\infty \mu(s) A\xi(s)\d s +f(w)=
-[f(u_\eps)-f(u_0)],\\
\xi_t=T\xi+\pt w,
\end{cases}
$$
with initial conditions $W(\tau)=0$.
By the usual multiplications, we get
$$
\frac{\d}{\d t}\|W\|_\H^2
\leq \|\pt w\|^{2}+c\|f(u_{\eps})-f(u_{0})\|^{2}.
$$
Exploiting \eqref{crescita} and \eqref{s4r9a}, we readily obtain
$$
\|f(u_{\eps})-f(u_{0})\| \leq c\| \bar{u}\|_{1}
\leq c \|w\|_{1}+c\|v\|_{1}.
$$
On the other hand, we know from~\eqref{vpicpic} that (note that $\vartheta\geq 1$)
$$
\|V(t)\|_{\H}\leq c\eps ^{1-\rho},\quad\forall t\geq \tau.
$$
Combining the estimates, we end up with
$$
\ddt \|W\|_{\H}^2
\leq c\|W\|_{\H}^2
+c\eps^{2(1-\rho)},
$$
and the Gronwall lemma yields
$$
\|W(t)\|_{\H}^2\leq \eps^{2(1-\rho )}
c\e^{c(t-\tau )},\quad\forall t\geq \tau.
$$
The desired conclusion follows then by comparison.
\end{proof}

\begin{proof}[Proof of Theorem \ref{s4the1}]
For $\eps>0$, let $U_{\eps }$ be an arbitrary element of $\A^\eps$.
Then $U_\eps=U_\eps(0)$ for some
{\cbt} $U_\eps(t)$ of
$S_{\hat g^\eps}(t,\tau)$.
By applying Lemma~\ref{s4pro1} with $t=0$,
$$
\|U_\eps
-S_{\hat{g}^{0}}(0,\tau)U_\eps(\tau)\| _{\H}
\leq c\eps ^{1-\rho}\e^{-c\tau},\quad\forall \tau\leq 0.
$$
At the same time, in the light of Remark~\ref{miii}, the set $\A^{0}$
attracts uniformly not only
with respect to $\tau \in \R$, but also with respect to
$\hat{g}^{0}\in \textsf{H}(g^{0})$.
Thus, setting $\nu>0$ arbitrarily small, and recalling~\eqref{s4r9a},
we find $\tau=\tau(\nu)\leq 0$ independent of $\eps$
such that
$$
\mathrm{dist}_{\H}\big(S_{\hat{g}^{0}}(0,\tau)U_\eps(\tau),\A^{0}\big)\leq\nu.
$$
Exploiting the triangle inequality we arrive at
$$\mathrm{dist}_{\H}\big(U_\eps,\A^{0}\big)
\leq c\eps ^{1-\rho}\e^{-c\tau}+\nu,
$$
and by arbitrariness of $U_\eps\in\A^\eps$ we reach the conclusion
$$\mathrm{dist}_{\H}\big(\A^\eps,\A^{0}\big)
\leq c\eps ^{1-\rho}\e^{-c\tau}+\nu.
$$
Accordingly,
$$\limsup_{\eps\to 0}\,\mathrm{dist}_{\H}\big(\A^\eps,\A^{0}\big)
\leq \nu.
$$
A final limit $\nu\to 0$ completes the argument.
\end{proof}



\end{document}